\documentclass[12pt]{article}

\usepackage{epsfig,amsmath,amssymb,graphics,verbatim,amsfonts,subfigure,psfrag,amsthm,color}
\usepackage[letterpaper,margin=1in]{geometry}

\usepackage{enumitem}
\makeatletter
\newcommand{\mylabel}[2]{#2\def\@currentlabel{#2}\label{#1}}
\makeatother

\let\oldbibliography\thebibliography
\renewcommand{\thebibliography}[1]{%
  \oldbibliography{#1}%
  \setlength{\itemsep}{0.5mm}%
}

\newtheorem{thm}{Theorem}[section]

\newtheorem{lem}[thm]{Lemma}

\newtheorem{defn}[thm]{Definition}

\def\R{\mathbb{R}}
\def\E{\mathbb{E}}

\def\P{\mathbb{P}}
\def\C{\mathbb{C}}
\def\N{\mathbb{N}}
\def\Z{\mathbb{Z}}
\def\I{\infty}
\def\Id{{\textnormal{Id}}}
\def\Cov{{\textnormal{Cov}}}
\def\Tr{{\textnormal{Tr}}}

\def\txtd{{\textnormal{d}}}
\def\txte{{\textnormal{e}}}

\def\txtl{{\textnormal{l}}}

\def\txtD{{\textnormal{D}}}

\newcommand{\be}{\begin{equation}}
\newcommand{\ee}{\end{equation}}
\newcommand{\bea}{\begin{eqnarray}}
\newcommand{\eea}{\end{eqnarray}}
\newcommand{\beann}{\begin{eqnarray*}}
\newcommand{\eeann}{\end{eqnarray*}}
\newcommand{\benn}{\begin{equation*}}
\newcommand{\eenn}{\end{equation*}}

\DeclareMathSymbol{\leqsymb}{\mathalpha}{AMSa}{"36}

\DeclareMathSymbol{\geqsymb}{\mathalpha}{AMSa}{"3E}

\def\ra{\rightarrow}
\def\I{\infty}

\newcommand{\cA}{{\mathcal A}}  
\newcommand{\cB}{{\mathcal B}}  
\newcommand{\cC}{{\mathcal C}}  
\newcommand{\cD}{{\mathcal D}}  
\newcommand{\cF}{{\mathcal F}}  
\newcommand{\cG}{{\mathcal G}}  
\newcommand{\cH}{{\mathcal H}}  
\newcommand{\cI}{{\mathcal I}}  
\newcommand{\cJ}{{\mathcal J}}  
\newcommand{\cL}{{\mathcal L}}  
\newcommand{\cO}{{\mathcal O}}  
\newcommand{\cS}{{\mathcal S}}  
\newcommand{\cU}{{\mathcal U}}  
\newcommand{\cW}{{\mathcal W}}  
\newcommand{\cZ}{{\mathcal Z}}  

\newcommand{\fr}{{\mathfrak r}}  

\begin{document}

\title{Combined Error Estimates\\ for Local Fluctuations of SPDEs}
\author{Christian Kuehn and Patrick K\"urschner}
\date{\today}   

\maketitle

\begin{abstract}
In this work, we study the numerical approximation of local fluctuations of
certain classes of parabolic stochastic partial differential equations (SPDEs). 
Our focus is on effects for small spatially-correlated noise on a time scale 
before large deviation effects have occurred. In particular, we are interested 
in the local directions of the noise described by a covariance operator. We 
introduce a new strategy and prove a Combined ERror EStimate (CERES) for the 
five main errors: the spatial discretization error, the local linearization 
error, the noise truncation error, the local relaxation error to steady state, 
and the approximation error via an iterative low-rank matrix algorithm. In 
summary, we obtain one CERES describing, apart from modelling of the original 
equations and standard round-off, all sources of error for a local fluctuation 
analysis of an SPDE in one estimate. To prove our results, we rely on a 
combination of methods from optimal Galerkin approximation of SPDEs, covariance 
moment estimates, analytical techniques for Lyapunov equations, iterative
numerical schemes for low-rank solution of Lyapunov equations, and working
with related spectral norms for different classes of operators.
\end{abstract}

\textbf{Keywords:} stochastic partial differential equation, stochastic dynamics, 
combined error estimates, optimal regularity, Lyapunov equation, low-rank 
approximation, local fluctuations.

\section{Introduction}
\label{sec:intro}

This work has two main goals. The first - more abstract - goal is to establish 
a general strategy to find and prove Combined ERror EStimates (CERES)\footnote{CERES 
is not only a direct acronym for the main strategy of this work but also contains a
nice historical note: The asteroid Ceres was predicted by Gau\ss~using the method
of least-squares, which also 'combines' the minimization of several errors to a 
single function.} for dynamical systems involving several sources of error. 
The second - more specific - goal is to demonstrate CERES for a concrete challenge 
of an infinite-dimensional stochastic problem. The technically precise formulation
of our work starts in Section~\ref{sec:SPDE}. In this introduction we provide a 
basic overview of our strategy and our main results. We focus on the evolution 
equation  
\be
\label{eq:SPDEintro}
\txtd u=\left[Au+f(u)\right]~\txtd t + g(u)~\txtd W,
\ee
where $W=W(x,t)$ is a certain white-noise process, $A$ is a suitable linear operator, 
$f,g$ are given maps and $u=u(x,t)\in\R$ is the unknown function~\cite{DaPratoZabczyk}. As
a paradigmatic example, one may think of the Laplacian $A=\Delta$, $W$ as a $Q$-Wiener 
process with trace-class operator $Q$ and $f,g$ as sufficiently smooth Lipschitz 
functions. Suppose the deterministic problem, i.e.~$g(u)\equiv 0$, has a steady state 
solution $u=u^*$, which solves $0=Au^*+f(u^*)$. Suppose the steady state is linearly 
stable, which just means that the spectrum of the linear operator 
\be
\label{eq:linAintro}
\tilde{A}:=A+\txtD_uf(u^*) 
\ee
is properly contained in the left-half of the complex plane~\cite{Kielhoefer,Robinson1}.
Note that since $f$ is scalar-valued, we may also write $\txtD_uf(u^*)=f'(u^*)\Id$ so 
that if $A$ is symmetric then so is $\tilde{A}$. 
Assume that the initial condition $u(x,0)$ is very close to $u^*$ and the 
noise in~\eqref{eq:SPDEintro} is sufficiently small, $0<\|g(u)\|=:\psi\leq 1$, 
in comparison to the spectral gap of $\tilde{A}$ to the imaginary axis. Then 
the probability is extremely large that one observes only local fluctuations 
of sample paths of the SPDE~\eqref{eq:SPDEintro} near $u^*$ on very long time 
scales. Only when the time scale reaches roughly order $\cO(\txte^{c/\psi^2})$ 
for some constant $c>0$ as $\psi\ra 0$, large deviation 
effects~\cite{WentzellFreidlin,DaPratoZabczyk} occur. Here we focus on the 
initial scale of fluctuations, which we refer to as sub-exponential scale, on which
large deviation effects do not play a role. However, the noise 
does still play an important role near the steady state. Its interplay with
the operator $\tilde{A}$ determines the directions, in which we are going to 
find the process with higher probability locally near $u^*$. This raises the 
question, how to numerically compute these directions. A recent practical 
strategy suggested in the context of a numerical 
continuation~\cite{KrauskopfOsingaGalan-Vioque} framework for stochastic 
ordinary differential equations (SODEs)~\cite{KuehnSDEcont1} and then extended 
for SPDEs~\cite{KuehnSPDEcont} is to: 

\begin{itemize} 
 \item[(S1)] spatially discretize $u$ with approximation level $h$ and consider 
the resulting SODEs for $u_h\in\R^N$. Note that we are going to view $u_h$ as 
a vector but also use the notation for the associated function expressed via
the basis of a finite-dimensional spatial approximation space;
 \item[(S2)] locally linearize the SODEs around $u^*_h\approx u^*$ and consider SODEs for 
the linear approximation $\tilde{U}_h\in\R^N$, which form an Ornstein-Uhlenbeck (OU) 
process~\cite{Gardiner};
 \item[(S3)] truncate the noise term based upon the decay of the $Q$-Wiener process
and consider a reduced linear OU process $U_h\in\R^N$;
 \item[(S4)] take the covariance matrix $V_h=V_h(t)$ of the OU process $U_h$,
note that $V_h$ satisfies a time-dependent Lyapunov equation~\cite{Antoulas,Simoncini}, 
and show that $V_h$ converges quickly in time to a stationary Lyapunov equation for a 
matrix $V_*\in\R^{N\times N}$;
 \item[(S5)] compute a low-rank approximation $V_*\approx \cZ\cZ^\top$ 
using a specialized iterative methods to generate $\cZ\in\R^{N\times \fr}$ with $\fr\ll N$; 
$j$ computation steps of an iterative algorithm yield a matrix $V_j\approx V_*$.
\end{itemize}

Every step (S1)-(S5) produces an error, i.e., the final matrix $V_j$ only provides 
an approximation to the infinite-dimensional covariance operator 
$\textnormal{Cov}(u)$~\cite{GoldysVanNeerven}, which is precisely the 
operator describing the different fluctuation directions. In summary, one should 
aim for a result of the form
\be
\label{eq:goal}
\sup_{t\in[0,T]}\|\textnormal{Cov}(u)-V_j\|\leq 
\substack{\text{error in} \\\text{Step (S1)}}  +
\substack{\text{error in} \\\text{Step (S2)}}  +
\substack{\text{error in} \\\text{Step (S3)}}  +
\substack{\text{error in} \\\text{Step (S4)}}  +
\substack{\text{error in} \\\text{Step (S5)}}  
\ee 
to really judge the quality from the viewpoint of numerical analysis. In
~\eqref{eq:goal} and similar comparison problems we are always going to 
view finite-dimensional operators such as $V_j$
as infinite-dimensional operators by using an embedding via a basis of the
function space on which $\textnormal{Cov}(u)$ is defined. 
Equation~\eqref{eq:goal} is just a prototypical example, i.e., a chain of 
different error terms does occur in most challenging high-dimensional
problems, particularly those involving stochastic aspects.\medskip 

{\small \textbf{Remark:} \textit{We do not consider in~\eqref{eq:goal} 
the modelling error of SPDEs of the form~\eqref{eq:SPDEintro}. On the 
one hand, it is partially included in a stochastic formulation anyway, 
and on the other hand, it is always possible to argue in an application, 
whether other terms or effects matter. The second error we do not include 
in the CERES is the standard numerical round-off error, which is universal 
for a given precision.}}\medskip

From a technical viewpoint, each step demands different techniques and then
a combination of the different estimates. For this CERES, we decided to 
consider the spectral or $2$-norm $\|\cdot\|_2$ on the infinite- as well
as finite-dimensional levels as well as the associated derived operator 
norm. This simplifies computing a CERES considerably
as one may use the standard triangle inequality and suitable embeddings for 
the expression 
\be
\|\underbrace{\textnormal{Cov}(u)-\textnormal{Cov}(u_h)}_{\textnormal{Step (S1)}}
+\underbrace{\textnormal{Cov}(u_h)-\textnormal{Cov}(\tilde{U}_h)}_{\textnormal{Step (S2)}}
+\underbrace{\textnormal{Cov}(\tilde{U}_h)-V_h}_{\textnormal{Step (S3)}}
+\underbrace{V_h-V_*}_{\textnormal{Step (S4)}}
+\underbrace{V_*-V_j}_{\textnormal{Step (S5)}}\|.
\ee
For (S1), we rely on an extension to covariance operators of optimal error
estimates for Galerkin finite elements methods for SPDEs~\cite{Kruse,KruseLarsson}.
(S2) is treated via small noise approximation in combination with moment
equations~\cite{Socha,KuehnMC}. (S3) is covered by standard growth estimates
for $Q$-Wiener processes over a finite time scale~\cite{DaPratoZabczyk}. (S4) is 
tackled by results on spectra for Lyapunov equations~\cite{Bellman} and decay of the
time-dependent problem~\cite{HuLiu}. (S5) requires a careful
tracing of error estimates for low-rank versions of iterative algorithms, such 
as alternating direction implicit (ADI)~\cite{LiWhite} and rational Krylov
methods~\cite{DruskinKnizhnermanSimoncini}. 

Our final result for~\eqref{eq:goal} is summarized in Theorem~\ref{thm:main}. It
illustrates that many factors can influence the error. For example, the spatial 
resolution $h$, the final time $T$, the Lipschitz constants of $f,g$, the 
structure of the operator $Q$, the spectrum of $\tilde{A}$, the noise truncation
level $R$, and the low-rank $\fr$ all appear in some form in the final error. 
Therefore, \emph{balancing} a CERES is the key practical message of our work. 
Just making a spatial resolution $h$ small or a dynamical error small by taking 
higher-order terms into account may not be enough in practice, i.e., one has to 
be aware, which error term dominates a CERES.\medskip

We highlight that extensive numerical continuation calculations, practical
convergence tests, as well as large-scale examples already exist for the method
to approximate local fluctuations numerically as proposed here. We refer the 
interested reader to~\cite{KuehnSDEcont1} for an introduction and the SODE case,
to~\cite{KuehnSPDEcont} for the SPDE case including precise numerical comparisons to
theoretical scaling laws, and to~\cite{Baarsetal} for a successful large-scale
application in geoscience. In this work, we focus on the theoretical aspects,
thereby finishing/completing the previous work on scientific computing side of
the method.\medskip 

The paper is structured as follows: In Section~\ref{sec:SPDE} we provide the
foundational technical setup for SPDEs~\eqref{eq:SPDEintro} using mild
solutions and in Section~\ref{sec:discretize} we cover the spatial 
discretization. Section~\ref{sec:SPDEcov} develops the error estimates for 
covariance matrices while Section~\ref{sec:SODE} contains the relevant 
moment estimates. Then we discuss the noise truncation in Section~\ref{sec:truncate}. 
In Section~\ref{sec:Lyapunov}, we transition to the Lyapunov
equation and its reduction to steady state. The last technical step is 
carried out in Section~\ref{sec:iteration} tracing the error results for
low-rank iterative schemes for Lyapunov equations. In Section~\ref{sec:result},
we present the full CERES, which draws upon all the previous results. An
outlook to open problems and further applications of our methodology is
given in Section~\ref{sec:outlook}.

\section{SPDE - Mild Solutions}
\label{sec:SPDE}

Let $\cD\subset \R^d$ be a bounded domain with smooth boundary $\partial \cD$ and denote the 
spatial variable by $x\in\cD$. Consider a fixed compact time interval $\cI=[0,T]$, $t\in\cI$ 
and a filtered probability space $(\Omega,\cF,(\cF_t)_{t\in\cI},\P)$. Furthermore, we fix two Hilbert
spaces $\cH$ and $\cU$ with inner products $\langle\cdot,\cdot\rangle_\cH$ and 
$\langle\cdot,\cdot\rangle_\cU$ respectively. Let $Q\in \cL(\cU,\cU)$ be a symmetric 
non-negative linear operator and let $(W(t))_{t\in \cI}$ denote the associated $Q$-Wiener process;
see~\cite[Section~4.1]{DaPratoZabczyk} for the definitions when $Q$ is of trace class 
and~\cite[Section~4.3]{DaPratoZabczyk} for the case of a cylindrical Wiener process. Let 
$\cU_0:=Q^{1/2}\cU$ be a Hilbert space~\cite{DaPratoZabczyk,PrevotRoeckner} with the 
inner product $\langle\cdot,\cdot\rangle_{\cU_0}:=\langle Q^{-1/2}\cdot, 
Q^{-1/2}\cdot\rangle_\cU$, where $Q^{-1/2}$ is the Moore-Penrose
pseudoinverse of $Q$. An operator $M:\cU_0\ra \cH$ is a Hilbert-Schmidt operator if
the norm
\benn
\|M\|_{\cL_2^0}:=\left(\sum_{k=1}^\I \|M\zeta_k\|_{\cH}^2\right)^{1/2}
\eenn
is finite, where the choice of orthonormal basis $\{\zeta_k\}_{k=1}^\I$ for $\cU_0$ turns 
out to be arbitrary. The space of these Hilbert-Schmidt operators will be denoted accordingly 
by $\cL_2^0$. If we consider Hilbert-Schmidt operators on $\cH$, then they will be 
denoted by $\cL_2=\cL_2(\cH,\cH)$.

Let $u(x,t;\omega)$ for $\omega\in\Omega$ denote the unknown family of random variables 
$u:\cD\times \cI\times \Omega\ra \R$. In the notation we shall always suppress $\omega$
from now on and assume that all maps we define are measurable with respect to $\omega$, 
which will also imply measurability for $u$ below. We write $u(t)=u(\cdot,t)\in \cH$, 
$W=W(\cdot,t)$, and 
study the SPDE
\be
\label{eq:SPDE}
\txtd u(t)=\left[Au(t)+f(u(t))\right]~\txtd t + g(u(t))~\txtd W(t),\qquad u(0)=u_0,
\ee
as an evolution equation on the Hilbert space $\cH$, which is taken as a suitable function
space on $\cD$. 

\begin{enumerate}
 \item[\mylabel{a:operator}{(A0)}] We assume that $Q$ is of trace class. Furthermore, 
we assume the operator $A:\textnormal{dom}(A)\subset 
\cH\ra \cH$ is linear, self-adjoint negative definite operator with compact inverse and 
generates an analytic semigroup $t\mapsto \txte^{tA}$~\cite{Henry}. 
\end{enumerate}

(A0) implies that there exists an orthonormal basis of eigenvectors of $A$ for $\cH$. 
The maps $f,g$ are going to be specified more precisely below and the initial condition 
$u_0\in\cH$ is a random variable. A mild solution $u(t)$ to~\eqref{eq:SPDE} satisfies
\be
\label{eq:SPDEmild}
u(t)=\txte^{tA}u_0+\int_0^t \txte^{(t-s)A}f(u(s))~\txtd s+\int_0^t\txte^{(t-s)A} 
g(u(s))~\txtd W(s),
\ee
i.e., the integral equation~\eqref{eq:SPDEmild} holds $\P$-almost surely ($\P$-a.s.) 
for $t\in\cI$ and 
\benn
\P\left(\int_0^t |u(s)|^2~\txtd s<+\I\right)=1, \qquad \text{$\P$-a.s.;}
\eenn
see also~\cite[Chapter~4]{DaPratoZabczyk} or~\cite{DalangQuerSardanyons} for the construction 
of the stochastic integral. It is well-known that under certain Lipschitz 
assumptions~\cite[Section~7.1]{DaPratoZabczyk} or dissipativity 
assumptions~\cite[Section~7.4.2]{DaPratoZabczyk} on $f,g$, there exists a unique mild solution.
However, since we are interested in numerical error estimates, it is important to have optimal
regularity results for mild solutions so we follow~\cite{Kruse,KruseLarsson,JentzenRoeckner1}.
Denote by $\{a_k\}_{k=1}^\I$, $0>a_1\geq a_2\geq \cdots$ the eigenvalues of $A$ and by 
$e_k$ the associated eigenfunctions with $A e_k=a_ke_k$. For $r\in\R$, define the fractional 
operator $A^{r/2}:\textnormal{dom}(A^{r/2})\ra \cH$ by
\benn
A^{r/2}v:=-(-A)^{r/2}v=-\sum_{k=1}^\I (-a_k)^{r/2}~ \langle v,e_k\rangle_\cH~ e_k.
\eenn
Set $\dot{\cH}^r:=\textnormal{dom}(A^{r/2})$ and consider the norm $\|v\|_{\dot{\cH}^r}
:=\|A^{r/2}v\|_\cH$, which turns $\dot{\cH}^r$ into a Hilbert space. Let 
$\cL^0_{2,r}\subset \cL_2^0$ denote the subspace of Hilbert-Schmidt operators, which have
finite norm $\|\cdot \|_{\cL^0_{2,r}}:=\|A^{r/2}\cdot\|_{\cL_2^0}$. The following 
assumptions are assumed to hold from now on (although we are still going to emphasize this 
several times in statements of theorems below):

\begin{itemize}
 \item[\mylabel{a:initial}{(A1)}] Fix two constants $r\in[0,1)$ and $p\in[2,\I)$ for all 
 assumptions. Fix an initial condition $u_0:\Omega\ra\dot{\cH}^{r+1}$ and assume 
it is $\cF_0$-measurable with 
 \benn
 \left[\E\left(\|u_0\|_{\dot{\cH}^{r+1}}^p\right)\right]^{1/p}\leq C_{\textnormal{ini},r}<+\I.
 \eenn
 \item[\mylabel{a:f}{(A2)}] $f:\cH\ra \dot{\cH}^{r-1}$ and there exists a constant $C_f>0$ 
 such that
 \be
 \label{eq:Lipf}
 \|f(u)-f(v)\|_{\dot{\cH}^{r-1}}\leq C_f\|u-v\|_\cH,\quad \text{for all $u,v\in\cH$.}
 \ee
 \item[\mylabel{a:g}{(A3)}] $g:\cH\ra \cL_2^0$ and there exists a constant $C_{g,1}>0$ 
 such that
 \be
 \label{eq:Lipg1}
 \|g(u)-g(v)\|_{\cL_2^0}\leq C_{g,1}\|u-v\|_\cH,\quad \text{for all $u,v\in\cH$,}
 \ee
 and furthermore $g(\dot{\cH}^r)\subset \cL^0_{2,r}$ holds with the estimate
  \be
 \label{eq:Lipg2}
 \|g(u)\|_{\cL_{2,r}^0}\leq C_{g,2}(1+\|u\|_{\dot{\cH}^r}),\quad \text{for all $u\in\cH$,}
 \ee
 and some constant $C_{g,2}>0$.
\end{itemize}

Essentially \ref{a:f}-\ref{a:g} are modifications/extensions of the classical Lipschitz 
assumptions in~\cite{DaPratoZabczyk}. Therefore, one immediately gets:

\begin{thm}[{\cite[Theorem~7.4]{DaPratoZabczyk},\cite[Theorem~1]{JentzenRoeckner}}]
\label{thm:existence}
Suppose \ref{a:operator}-\ref{a:g} hold, then there 
exists a unique mild solution for the SPDE~\eqref{eq:SPDE}. 
\end{thm}

In addition, one may provide (optimal) regularity estimates for the mild solution:

\begin{thm}[{\cite[Theorem~3.1~\&~Theorem~4.1]{KruseLarsson}}]
\label{thm:regularity}
Suppose \ref{a:operator}-\ref{a:g} hold then the unique mild solution 
is almost surely in $\dot{\cH}^{r+1}$. There exists a constant $C_{\textnormal{spa}}>0$ 
such that
\be
\label{eq:spacereg}
\sup_{t\in\cI}\left(\E\left[\|u(t)\|^p_{\dot{\cH}^{r+1}}\right]\right)^{1/p}\leq 
C_{\textnormal{ini},r} + C_{\textnormal{spa}}\left(1+
\sup_{t\in\cI}\left(\E\left[\|u(t)\|^p_{\dot{\cH}^{r}}\right]\right)^{1/p}\right)
\ee
holds as a spatial regularity estimate and for temporal regularity we have
\be
\label{eq:tempreg}
\sup_{t_1,t_2\in\cI,t_1\neq t_2} \left(\E\left[\|u(t_1)-
u(t_2)\|_{\dot{\cH}^{s}}^p\right]\right)^{1/p}\leq 
C_{\textnormal{time}}|t_1-t_2|^{\min(\frac12,\frac{1+r-s}{2})}
\ee
for some constant $C_{\textnormal{time}}>0$ and for every $s\in[0,r+1)$.
\end{thm}

As for the deterministic counterpart ($g\equiv0$) of the SPDE~\eqref{eq:SPDE}, one may
use regularity estimates to obtain convergence rates and error estimates for associated 
numerical schemes~\cite{Thomee,ErnGuermond}. In this paper, as already discussed in 
Section~\ref{sec:intro}, we are interested in the spatial approximation error only. 
We stated the temporal regularity results only for 
completeness. It is noted that the constant $C_{\textnormal{spa}}>0$ in 
Theorem~\ref{thm:regularity} does depend upon $p,r,A,f,g$ and $T$.\medskip

As a typical example for the abstract framework, one should always keep in mind
classical one-component reaction-diffusion systems, where
\benn
\cH=L^2(\cD)\qquad \text{and} \qquad A=\Delta:=\sum_{k=1}^d\frac{\partial^2}{\partial x_k^2},
\eenn
with Dirichlet boundary conditions, but we shall not restrict to this case in this paper 
and continue to work with the more general setup.

\section{SPDE - Spatial Discretization}
\label{sec:discretize}

Let $\cS_h$ for $h\in(0,1]$ denote a continuous family of finite-dimensional subspaces 
of the Hilbert space $\dot{\cH}^1$ with
\be
\dim(\cS_h)=:N\in\N,
\ee
which are spanned by $N$ basis elements of an orthnormal basis of $\dot{\cH}^1$. The spaces 
$\cS_h$ are going to be used as the spatial discretization spaces; see~\cite{Thomee,ErnGuermond} 
for a detailed overview. Particularly important examples 
are the span of a finite set of basis functions of $A$ leading to a spectral Galerkin method, 
or the span of piecewise polynomial functions on $\cD$, where $h$ is the diameter of the 
largest element of a suitable mesh on $\cD$, leading to a Galerkin finite element method. 
Let $R_h:\dot{\cH}^1\ra \cS_h$ be the orthogonal projection onto $\cS_h$ with 
respect to the inner product $a(\cdot,\cdot):=\langle A^{1/2}\cdot,A^{1/2}\cdot\rangle_\cH$
so that we have
\benn
a(R_hu,v_h)=a(u,v_h),\quad \text{for all $u\in\dot{\cH}^1$, $v_h\in\cS_h$.}
\eenn
The discretized version $A_h$ of $A$ is defined by the requirement that for a given 
$u_h\in\cS_h$, the image $A_hu_h$ is the unique element satisfying
\benn
a(u_h,v_h)=\langle A_hu_h,v_h\rangle_\cH,\quad \text{for all $v_h\in\cS_h$.} 
\eenn
Furthermore, let $P_h:\dot{\cH}^{-1}\ra \cS_h$ be the orthogonal projection onto $\cS_h$
defined analogously to $R_h$, i.e., by requiring
\benn
\langle P_hu,v_h\rangle_\cH=a(A^{-1}u,v_h),\quad \text{for all $u\in\dot{\cH}^{-1}$, 
$v_h\in\cS_h$.}
\eenn
Then the spatial discretization of the SPDE can be written as 
\be
\label{eq:SPDEdiscrete}
\txtd u_h(t)=\left[A_hu_h(t)+P_hf(u_h(t))\right]~\txtd t + P_hg(u_h(t))~\txtd W(t),
\qquad u_h(0)=P_hu_0.
\ee
It is relatively straightforward to check that~\eqref{eq:SPDEdiscrete} must also
have a unique mild solution; see also the proof of Lemma~\ref{lem:sum}. To fix the 
role of the discretization parameter $h$, we are going to assume:
\begin{itemize}
 \item[\mylabel{a:project}{(A4)}] There exists a constant $C_{h}>0$ such that
 \be
 \label{eq:Rcond}
 \|R_h v-v\|_{\cH}\leq C_{h}~ h^s~\|v\|_{\dot{\cH}^s},\quad \text{for all 
 $v\in\dot{\cH}^s$, $s\in\{1,2\}$, $h\in(0,1]$.}
 \ee
\end{itemize}

Recently, the regularity result of Theorem~\ref{thm:regularity} has been transferred
to yield results on strong/pathwise approximation properties of the approximating 
stochastic evolution equation~\eqref{eq:SPDEdiscrete} for the SPDE~\eqref{eq:SPDE}. Consider
the norm
\be
\|\cdot\|_{L^p(\Omega;\cH)}:=\left(\E[\|\cdot\|_\cH^p]\right)^{1/p}
\ee
for $p\in[2,+\I)$ as above. Then one can prove the following error estimate:

\begin{thm}[{\cite[Theorem~1.1]{Kruse}}]
\label{thm:approxKruse}
Suppose \ref{a:operator}-\ref{a:project} hold, then there exists a constant 
$C_{\textnormal{Gal}}>0$ such that 
\be
\label{eq:approxKruse}
\|u(t)-u_h(t)\|_{L^p(\Omega;\cH)}\leq C_{\textnormal{Gal}}~h^{1+r},\qquad \forall t\in\cI.
\ee
\end{thm}

In the later development of the error estimates for the covariance operator, we shall
need another auxiliary result, which we prove here:

\begin{lem}
\label{lem:sum}
Suppose \ref{a:operator}-\ref{a:project} hold, then here exists a constant $C_+>0$ (independent 
of $h$) such that 
\benn
\sup_{t\in\cI} \|u(t)+u_h(t)\|_{L^2(\Omega;\cH)}\leq C_+.
\eenn
\end{lem}

\begin{proof}
Indeed, we just have
\benn
\sup_{t\in\cI}\|u(t)+u_h(t)\|_{L^2(\Omega;\cH)}\leq \sup_{t\in\cI}\|u(t)\|_{L^2(\Omega;\cH)}
+\sup_{t\in\cI} \|u_h(t)\|_{L^2(\Omega;\cH)}
\eenn
so the first term is bounded by a direct application of Theorem~\ref{thm:regularity} (or in
fact, the classical result~\cite[Theorem~7.4(ii)]{DaPratoZabczyk}). For the discretization, note
that we may apply the same results since (A0)-(A3) also hold for the discretized 
SPDE~\eqref{eq:SPDEdiscrete}. For example, consider (A1) then we have
\be
\|P_hf(u)-P_hf(v)\|_{\dot{H}^{r-1}}=\|P_h[f(u)-f(v)]\|_{\dot{H}^{r-1}}\leq 
\|f(u)-f(v)\|_{\dot{H}^{r-1}}
\ee
since $P_h$ is an orthogonal projector. The other assumptions are checked similarly.
\end{proof}

Of course, it should be noted that the constant $C_+>0$ will depend on the data of 
the problem, i.e., on $f,g,T,A$. However, we will only be interested in the convergence
rate in $h$ and we will show later on that we can select $T$ as an order one constant anyhow,
while $f,g,A$ are given and satisfy the Lipschitz and semigroup assumptions stated above.
For more background on discretization of SPDEs, we also refer 
to~\cite{JentzenKloeden,LordPowellShardlow}.

\section{SPDEs - Covariance}
\label{sec:SPDEcov}

The next step is to establish numerical error estimates for 
the covariance. Let $v\in L^2(\Omega;\cH)$, then one defines the covariance 
operator~\cite{DaPratoZabczyk,LangLarssonSchwab} of $v$ as
\be
\label{eq:covSPDE}
\Cov (v):=\E[(v-\E[v])\otimes (v-\E[v])].
\ee
By definition, $\Cov (v):\cH\ra \cH$ is a symmetric linear operator. In addition one may check
that $\Cov (v)$ is nuclear using the equivalent characterization of nuclear operators $M$ on
Hilbert spaces via the condition
\benn
\Tr(M):=\sum_{k=1}^\I \langle M\xi_k , \xi_k\rangle_\cH <+\I
\eenn
for an orthonormal basis $\{\xi_k\}_{k=1}^\I$ of $\cH$. The space of nuclear operators 
$\cL_1(\cH,\cH)$ becomes a Banach space under the norm $\|\cdot\|_{\cL_1(\cH,\cH)}:=
\Tr(\cdot)$~\cite[Appendix~C]{DaPratoZabczyk}. Note that we have two well-defined covariance 
operators
\be
\Cov(u(t))\qquad \text{and}\qquad \Cov(u_h(t))
\ee
as the SPDE~\eqref{eq:SPDE} and the spatially discretized version~\eqref{eq:SPDEdiscrete}
both have mild solutions in $L^2(\Omega\times \cI;\cH)$. 

\begin{thm}
\label{thm:goal1}
There exists a constant $C_\txtd>0$ such that
\be
\sup_{t\in\cI}\|\Cov(u(t))-\Cov(u_h(t))\|_{\cL_2(\cH,\cH)}\leq C_\txtd h^{1+r}.
\ee
\end{thm}

\begin{proof}
The proof is a calculation aiming to use the spatial approximation property of 
Theorem~\ref{thm:approxKruse}. Suppose as a first case that 
$\E[u]=0=\E[u_h]$. Consider an orthonormal basis $\{\xi_k\}_{k=1}^\I$ of $\cH$.
We want to estimate the error in the Hilbert-Schmidt norm $\|\cdot\|_{\cL_2}$.
For the following steps we suppress the argument $u=u(t)$ and $u_h=u_h(t)$
and use the definition of the covariance operator
\beann
\sup_{t\in\cI}\|\Cov(u)-\Cov(u_h)\|_{\cL_2}^2&=&\sup_{t\in\cI} \|\E[u\otimes u]-\E[u_h\otimes u_h]
\|_{\cL_2}^2,\\ 
&\leq &\sup_{t\in\cI} \|\E[u\otimes u]-\E[u_h\otimes u_h]\|_{\cL_1}^2,
\eeann
where we used that the trace-class norm bounds the Hilbert-Schmidt norm. Hence,
we have using that off-diagonal terms vanish in the trace-class norm that
\beann
\sup_{t\in\cI}\|\Cov(u)-\Cov(u_h)\|_{\cL_2}^2
&\leq &\sup_{t\in\cI} \sum_{k=1}^\I \langle \E[(u-u_h)\otimes (u+u_h)] \xi_k,\xi_k\rangle^2_\cH,\\
&=&\sup_{t\in\cI} \sum_{k=1}^\I \langle \E[((u-u_h)\otimes (u+u_h) )\xi_k],\xi_k\rangle^2_\cH,\\
&=&\sup_{t\in\cI} \sum_{k=1}^\I \E[ \langle u-u_h,\xi_k\rangle_\cH \langle u+u_h,\xi_k\rangle_\cH]^2,\\
\eeann
where we use in the last step the definition of the tensor product 
$(v_1\otimes v_2)v_3:=\langle v_2,v_3\rangle_\cH v_1$ for $v_k\in\cH$ for $k\in\{1,2,3\}$. 
Then a direct application of Cauchy-Schwarz yields
\bea
\sup_{t\in\cI}\|\Cov(u)-\Cov(u_h)\|_{\cL_2}^2&\leq &\sup_{t\in\cI} \sum_{k=1}^\I 
\E[ \langle u-u_h,\xi_k\rangle_\cH^2] \E[\langle u+u_h,\xi_k\rangle_\cH^2],\nonumber\\
&\leq &\sup_{t\in\cI} \|u-u_h\|^2_{L^2(\Omega;\cH)} ~\| u+u_h\|_{L^2(\Omega;\cH)}^2.
\label{eq:lastline}
\eea
In the expression~\eqref{eq:lastline}, we may estimate the two terms separately by estimating 
the supremum by the product of suprema. Furthermore, a direct application of 
Theorem~\ref{thm:approxKruse} to the first term and Lemma~\ref{lem:sum} to the second term give 
\benn
\sup_{t\in\cI}\|\Cov(u)-\Cov(u_h)\|_{\cL_2}^2\leq C_{\textnormal{Gal}}~ C_{+}~ h^{2(r+1)},
\eenn
which yields the result in the basic case of zero means. If $\E[u]\neq0$ and $\E[u_h]\neq 0$,
we observe that 
\beann
\sup_{t\in\cI} \|u-u_h-\E[u-u_h]\|_{L^2(\Omega;\cH)}^2&\leq& 
\sup_{t\in\cI} \|u-u_h\|_{L^2(\Omega;\cH)}^2+ \sup_{t\in\cI}\|\E[u-u_h]\|_{L^2(\Omega;\cH)}^2,\\
&\leq& C_{\textnormal{Gal}}~ h^{2(r+1)} + \sup_{t\in\cI}\E[\|u-u_h\|_{L^2(\Omega;\cH)}]^2,\\
&\leq& 2 C_{\textnormal{Gal}}~ h^{2(r+1)},
\eeann
using again Theorem~\ref{thm:approxKruse} twice. Furthermore, it is easy to see that 
\benn
\E[\| u+u_h-\E[u+u_h]\|_{L^2(\Omega;\cH)}]
\eenn
remains bounded for nonzero means. The result now follows repeating the same steps shown
for the zero means case also for the nonzero means case.
\end{proof}

Obviously the constant $C_\txtd$ also depends upon the data of the problem as do $C_+$ and 
$C_{\textnormal{Gal}}$ but all constants are independent of $h$, which is the key discretization
parameter in the step (S1).

\section{SODEs - Linearization}
\label{sec:SODE}

Having estimated the error of the spatial discretization, we are now dealing with 
\be
\label{eq:SODE}
\txtd u_h(t)=\left[A_hu_h(t)+P_hf(u_h(t))\right]~\txtd t + P_h g(u_h(t))~\txtd W(t),
\qquad u_h(0)=P_hu_0,
\ee
where $u_h\in\R^{N}$ is a finite-dimensional approximation of $u$. We assume that the 
original SPDE~\eqref{eq:SPDE} and its projection 
have a locally asymptotically stable homogeneous steady state for zero noise and that we 
only study the small noise regime.

\begin{itemize}
 \item[\mylabel{a:spectrum}{(A5)}] Assume $u^*$ satisfies $Au^*+f(u^*)=0$. Furthermore, 
suppose $f$ is Fr\'echet differentiable, $\txtD_uf(u^*)=f'(u^*)\textnormal{Id}$ and 
\be
\textnormal{spec}(A+f'(u^*)\textnormal{Id})\subset \{\rho\in\R :\rho<0\}.
\ee
Furthermore, assume $u_h^*:=P_hu^*$ satisfies $A_hu_h^*+P_hf(u_h^*)=0$ for all $h\in(0,1]$.
 \item[\mylabel{a:smallnoise}{(A6)}] Suppose there exists a constant $\psi>0$ such that 
\be
\|P_hg(u_h(t))\|_2\leq \psi
\ee
for all $h\in(0,1]$.
\end{itemize} 

Below we are also going to assume that $\psi$ is chosen sufficiently small to get a 
good approximation of the linearized system. The goal is to provide a finite-time estimate 
for the difference between the covariance matrix $\Cov(u_h(t))$ of~\eqref{eq:SODE} and 
covariance matrix $\Cov(U_h(t))$ of the linearized OU process
\be
\label{eq:SODE1}
\txtd \tilde{U}_h(t)=\left[A_h+P_h[\txtD_uf](u^*)\right]\tilde{U}_h(t)~\txtd t + P_hg(u^*)~\txtd W(t),
\qquad \tilde{U}_h(0)=P_hu_0,
\ee
where $\txtD_uf$ denotes the usual Fr{\'e}chet derivative as introduced already 
above. Note that the linear operator $P_hg(u^*)$ only acts nontrivially on the first 
$N$ basis elements $\{e_i\}_{i=1}^N$ and is zero on $\{e_i\}_{i=N+1}^\I$. Hence, we 
can replace $W(t)$ by 
\benn
W^N(t)=\sum_{i=1}^N \sqrt{\lambda_{Q,i}}\beta_i(t)e_i
\eenn
where we assume that $\{e_i\}_{i=1}^\I$ are also eigenfunctions of $Q$, 
$\{\lambda_{Q,i}\}_{i=1}^\I$ are eigenvalues of $Q$, and $\{\beta_i(t)\}_{i=1}^\I$ are 
independent Brownian motions. Recall that without additional assumptions on the 
nonlinearity of $f$, it is usually not possible to estimate the error between a linearized 
and a nonlinear system, even on a finite time scale. To simplify the notation 
we let
\benn
u_h =: z,\qquad A_hu_h(t)+P_hf(u_h(t))=:F(z), \qquad P_hg(u_h(t))=:G(z),
\eenn
as well as
\benn
\tilde{U}_h =: \tilde{Z},\qquad \left[A_h+P_h[\txtD_uf](u^*)\right]=:\cA, \qquad P_hg(u^*)=:\tilde{B},
\eenn
where $G(z)$ and $\tilde{B}$ are operators projected/restricted onto the 
first $N$ basis functions. Hence, we have to compare the SODE
\be
\label{eq:brSODE}
\txtd z = F(z)~\txtd t+G(z)~\txtd W^N,\qquad z(0)=z_0,
\ee
near a steady state $z^*:=u_h^*$ to the SODE
\be
\label{eq:brSODElin}
\txtd \tilde{Z} = \cA \tilde{Z}~\txtd t+\tilde{B}~\txtd W^N,\qquad \tilde{Z}(0)=\tilde{Z}_0.
\ee
Without loss of generality we may assume that $z^*\equiv (0,\ldots,0)^\top$
since we can always translate the steady state if necessary. Let 
${\bf p}=(p_1,p_2,\ldots,p_N)\in(\N_0)^N$ be a multi-index, define the mean values
of $z$ by $\mu:=\E[z]\in\R^N$ and the centered moments as
\be
\E[(z-\mu)^{\bf p}]:=\E[(z_1-\mu_1)^{p_1}(z_2-\mu_2)^{p_2}\cdots (z_N-\mu_N)^{p_N}].
\ee
To make the notation more compact, we also introduce 'altered' multi-indices as follows:
\benn
{\bf p}(k:\zeta)=(p_1,p_2,\ldots,p_{k-1},p_{k}+\zeta,p_{k+1},\ldots,p_N)
\eenn
for $\zeta\in\Z$ and multiple arguments pertain to changes in the respective components.
Furthermore, we consider the diffusion operators
\be
\cG(z):=G(z)G(z)^\top,\qquad \tilde{\cB}:=\tilde{B}\tilde{B}^\top.
\ee

\begin{lem}
\label{lem:Socha}
The evolution equations for the centered moments $\E[(z-\mu)^{\bf p}]$ of~\eqref{eq:brSODE} 
are given by
\beann
\label{eq:cm}
\frac{\txtd}{\txtd t} \E[(z-\mu)^{\bf p}] &=& \sum_{k=1}^N p_k \E[F_k(z)(z-\mu)^{{\bf p}(k:-1)}]
+\frac12 \sum_{k=1}^Np_k(p_k-1)\E[\cG_{kk}(z)(z-\mu)^{{\bf p}(k:-2)}]\\
&& + \sum_{l=2}^N\sum_{k=1}^{l-1} p_lp_k \E[\cG_{kl}(z)(z-\mu)^{{\bf p}(k:-1,l:-1)}]
\eeann
\end{lem}

\begin{proof}
The calculation of the SODEs for the moments follows from first applying It\^o's formula
to monomials of the form $z^{\bf p}$ as shown in~\cite[Section~4.1]{Socha}. However, we 
then need to average via $\E[\cdot]$, and terms of the form $\int (\cdot)~\txtd W^N$ average
to zero as they satisfy the martingale property~\cite{Kallenberg} by the Lipschitz 
assumption on $G$. 
\end{proof}

Consider the linear approximation of~\eqref{eq:brSODE} given by~\eqref{eq:brSODElin}. We 
use the notation $\nu:=\E[\tilde{Z}]$ and we may assume without loss of generality that $\cA$ is 
already diagonal; indeed, by assumptions~\ref{a:operator} and \ref{a:spectrum} we 
already have that $\cA$ is symmetric with real spectrum so we can apply a coordinate 
change to make $\cA$ diagonal, which will just change constants in the estimates so we do 
not display this explicitly here. Now we have two processes and we would like 
to compare $\Cov(z)$ with $\Cov(\tilde{Z})$. From Lemma~\ref{lem:Socha} it follows
that 
\be
\frac{\txtd }{\txtd t}\mu = \E[F(z)],\qquad \frac{\txtd }{\txtd t}\nu = \cA\nu.
\ee
It is relatively easy to write down the formal evolution equations for the covariances. 
For the diagonal entries we have
\bea
\frac{\txtd }{\txtd t} \Cov(z)_{ii} &=& 2 \E[F_i(z)(z_i-\mu_i)] + \E[\cG_{ii}(z)] \\
\frac{\txtd }{\txtd t} \Cov(Z)_{ii} &=& 2 \E[(\cA \tilde{Z})_i(\tilde{Z}_i-\nu_i)] + \cB_{ii}
\eea
Consider the difference $\Cov(z)-\Cov(\tilde{Z})=:\Cov^\Delta$ and also define the remainders
\benn
R^F(z):=F(z)-\cA z,\qquad R^G(z):=G(z)-\tilde{\cB}.
\eenn
Observe that the remainder $R^F$ is Lipschitz 
\be
\|R^F(z_1)-R^F(z_2)\|\leq (C_F+\|\cA\|)\|z_1-z_2\|\label{eq:PLip1},
\ee
by assumptions~\ref{a:f} and \ref{a:spectrum} for some constant $C_F>0$, where $\|\cdot\|$ always
denotes the $2$-norm. Regarding the remainder $R^G$, recall that we assumed a uniform noise
bound in~\ref{a:smallnoise} so that 
\be
\|R^G(z_1)-R^G(z_2)\|\leq C_G\psi^2\label{eq:PLip2}
\ee
for some constant $C_G>0$. Then one finds, using that $\cA$ is diagonal and via 
assumption~\ref{a:spectrum}, that
\beann
\frac{\txtd }{\txtd t} \Cov^\Delta_{ii} &=& 2 \E[F_i(z)(z_i-\mu_i)-(\cA \tilde{Z})_i(\tilde{Z}_i-\nu_i)] 
+ \E[G_{ii}(z)-\cB_{ii}],\\
&=&2\cA_{ii} \E[(z_i+\mu_i-\mu_i-R^F_i(z)/\cA_{ii})(z_i-\mu_i)-(\tilde{Z}_i+\nu_i-\nu_i)(\tilde{Z}_i-\nu_i)] 
+ \E[R_{ii}^G(z)],\\
&=&2\cA_{ii}\left( \Cov^\Delta_{ii} + \E[(\mu_i-R^F_i(z)/\cA_{ii})(z_i-\mu_i)
-\nu_i(\tilde{Z}_i-\nu_i)] \right)+ \E[R_{ii}^G(z)],\\
&=&2\cA_{ii}\left( \Cov^\Delta_{ii} + \E[R^F_i(z)/\cA_{ii}(\mu_i-z_i)] \right)
+ \E[R_{ii}^G(z)].
\eeann
Using the Lipschitz conditions~\eqref{eq:PLip1} and the bound~\eqref{eq:PLip2} one has 
\bea
|\E[R^F_i(z)/\cA_{ii}(\mu_i-z_i)+R_{ii}^G(z)]|&\leq& \frac{(\|\cA\|+C_F)|\mu_i|}{|\cA_{ii}|}
\E[\|z\|] \nonumber\\
&&+\frac{(\|\cA\|+C_F)}{|\cA_{ii}|}\E[\|z\|^2]+C_G\psi^2. \label{eq:expected}
\eea
We denote the right-hand side of the last inequality by $\eta_{ii}(t)$, which implies 
the final estimate
\be
\label{eq:tbG}
\frac{\txtd }{\txtd t} \Cov^\Delta_{ii}(t)\leq \eta_{ii}^*+2\cA_{ii} 
\Cov^\Delta_{ii}(t),\qquad \eta^*_{ii}:=\max_{t\in \cI}\eta_{ii}(t).
\ee
Considering the coordinate shift $\Cov^\Delta_{ii}(t)=C^\Delta_{ii}(t)-
\eta^*_{ii}/{2\cA_{ii}}$ yields
\be
\label{eq:tbG1}
\frac{\txtd }{\txtd t} C^\Delta_{ii}(t)\leq 2\cA_{ii} 
C^\Delta_{ii}(t).
\ee
Applying Gronwall's inequality to~\eqref{eq:tbG1}, transforming back into 
original coordinate frame, and using $\eta^*_{ii}/(2\cA_{ii})<0$ yields the following result:

\begin{lem}
\label{lem:cii}
Suppose the assumptions \ref{a:operator}-\ref{a:smallnoise} hold for all $t\in\cI$, 
then
\be
\label{eq:estimateii}
\Cov^\Delta_{ii}(t)\leq -\frac{\eta^*_{ii}}{2\cA_{ii}}+
\left[\Cov^\Delta_{ii}(0)+\frac{\eta^*_{ii}}{2\cA_{ii}}\right]\txte^{2\cA_{ii}t}. 
\ee
holds for all $t\in\cI$.
\end{lem}

Note that an estimate of the form~\eqref{eq:estimateii} is fully expected to hold since
it states that the growth of the difference between the covariances in the case of 
Lipschitz $F$ and sufficiently bounded noise is controlled by the first- and second-moments
of the nonlinear process. In particular, if we are close to a linear SODE or the spectral
gap is very large, then we have an excellent finite-time approximation on $\cI$, while large
noise and a strong nonlinearity make the approximation worse. The next step is to look at the 
off-diagonal terms and consider the case $i>j$ (the case $i<j$ is similar). The evolution 
equations are
\beann
\frac{\txtd }{\txtd t} \Cov(z)_{ij} &=& \E[F_i(z)(z_j-\mu_j)]+\E[F_j(z)(z_i-\mu_i)] 
+ \E[G_{ij}(z)] \\
\frac{\txtd }{\txtd t} \Cov(\tilde{Z})_{ij} &=& \E[(\cA \tilde{Z})_i(\tilde{Z}_j-\nu_j)]
+\E[(\cA \tilde{Z})_j(\tilde{Z}_i-\nu_i)] 
+ \cB_{ij} 
\eeann
A similar calculation as above leads one to define 
\beann
\eta_{ij}(t)&=&\left(\frac{(\|\cA\|+C_F)|\mu_i|}{|\cA_{ii}|}+\frac{(\|\cA\|+C_F)|\mu_j|}{|\cA_{jj}|}
C_G\psi^2\right)\E[\|z\|]\\
&&+\left[\frac{(\|\cA\|+C_F)}{|\cA_{ii}|}+\frac{(\|\cA\|+C_F)}{|\cA_{jj}|}\right]\E[\|z\|^2].
\eeann
As before, we use the notation $\eta_{ij}^*:=\max_{t\in\cI}\eta_{ij}(t)$ and use 
Gronwall's inequality to obtain the result:

\begin{lem}
\label{lem:cij}
Suppose the assumptions \ref{a:operator}-\ref{a:smallnoise} hold for all $t\in\cI$, then
\be
\label{eq:estimateij}
\Cov^\Delta_{ij}(t)\leq -\frac{\eta^*_{ij}}{\cA_{ii}+\cA_{jj}}+
\left[\Cov^\Delta_{ij}(0)+\frac{\eta^*_{ij}}{\cA_{ii}+\cA_{jj}}\right]\txte^{(\cA_{ii}+\cA_{jj})t}.  
\ee
holds for all $t\in\cI$.
\end{lem}

Of course, the estimates~\eqref{eq:estimateii}-\eqref{eq:estimateij} may blow-up as $t\ra +\I$ 
if the stationary distribution cannot be approximated well by an OU process. Indeed, if the 
noise is small, this is exactly the effect of large 
deviations~\cite{WentzellFreidlin,DaPratoZabczyk}, 
which are going to occur on an asymptotic time scale $\cO(\txte^{c/\psi^2})$ as $\psi\ra 0$. 
However, the estimate is rather explicit, i.e., if we know the Lipschitz constant, the spectral
gap, the noise level and/or have some a-priori knowledge of the norms $\|z\|$ and/or $\|z\|^2$, 
then $\cA_{ij}<0$ is going to give decay for the exponential terms 
so that the only remaining term is $-\eta_{ij}^*/(\cA_{ii}+\cA_{jj})$. The linear 
approximation~\eqref{eq:brSODElin} will be a good approximation for a certain 
initial time-scale. The worst-case bound is the following:

\begin{thm}
\label{thm:goal2}
Suppose the assumptions \ref{a:operator}-\ref{a:smallnoise} hold for all $t\in\cI$, 
then there exists a constant $C_\txtl>0$ and a constant $\eta^*_\cA:=\max_{ij}-
\eta_{ij}^*/(\cA_{ii}+\cA_{jj})$ such that
\be
\|\Cov(z(t))-\Cov(\tilde{Z}(t))\|_2\leq \eta^*_\cA+C_\txtl [\|\Cov(z_0)-\Cov(\tilde{Z}_0)\|_2] 
\txte^{-t\min_i |\cA_{ii}|}.
\ee
\end{thm}

\begin{proof}
Using Lemma~\ref{lem:cii} and Lemma~\ref{lem:cij} the result easily follows.
\end{proof}

In summary, Theorem~\ref{thm:goal2} just states that one has to be extremely careful to
trust a local linear approximation in a numerical context for problems with large noise
and/or a very strong nonlinearity, which both lead to a very quick pathwise sampling of a 
non-Gaussian stationary distribution. However, for small noise, sub-exponential time scales
and/or a weak nonlinearity, the local approximation via linearization and covariance 
operators of an OU process is going to correctly expose the relevant directions of
fluctuations.

\section{SODEs - Noise Truncation}
\label{sec:truncate}

In many practical applications, one tends to make one further approximation which we also
discuss here. The main observation is that for sufficiently fast decay of the eigenvalues 
$\lambda_{Q,i}$, it is highly beneficial in practical computations to consider a further 
approximation  
\benn
W^N(t)\approx \sum_{i=1}^R \sqrt{\lambda_{Q,i}}\beta_i(t)e_i=:W^R(t)
\eenn
for some $R\leq N$, i.e., one truncates the noise. This means we now have to compare 
two OU process given by our previous linear SODE
\be
\label{eq:brSODElin1}
\txtd \tilde{Z} = \cA \tilde{Z}~\txtd t+\tilde{B}~\txtd W^N,\qquad \tilde{Z}(0)=\tilde{Z}_0.
\ee
as well as the noise-truncated SODE
\be
\label{eq:brSODElin1truncated}
\txtd U = \cA U~\txtd t+B~\txtd W^R,\qquad U(0)=\tilde{Z}_0.
\ee
Here $B$ denotes the reduction of the matrix $\tilde{B}\in\R^{N\times N}$ to the first 
$N\times R$ block. In particular, we have to compare $\textnormal{Cov}(\tilde{Z})$ and 
$\textnormal{Cov}(U)$. 

\begin{thm}
\label{thm:truncate}
Considering the SODEs~\eqref{eq:brSODElin1} and~\eqref{eq:brSODElin1truncated}
we have
\benn
\sup_{t\in\cI} \|\Cov(\tilde{Z})-\Cov(U)\|_2\leq C_{tr} \sum_{i=R+1}^N\lambda_{Q,i}
\eenn 
where $C_{tr}>0$ is a constant.
\end{thm}

\begin{proof}
As in Section~\ref{sec:SPDEcov}, we find that
\benn
\sup_{t\in\cI} \|\Cov(\tilde{Z})-\Cov(U)\|_2\leq C~ \sup_{t\in\cI} \|\Cov(\tilde{Z}-U)\|_2 
\eenn
where $C>0$ is some constant. Then we observe that the process $\tilde{Z}-U$
satisfies the SODE
\benn
\txtd (\tilde{Z}-U) = \tilde{B}~\txtd W^N-B~\txtd W^R=B_{N-R}\txtd W^{N-R}
\eenn
for a matrix $B_{N-R}\in\R^{(N-R)\times N}$ since the drift term and the first 
$R$ noise components are identical. Therefore, the result follow from the fact
that for a $Q$-Wiener process the covariance matrix is given by $tQ$.
\end{proof}

The main conclusion from Theorem~\ref{thm:truncate} is that we can decrease
the noise truncation error by making $R$ larger but that this yields a matrix $B$ of
higher rank if $R$ is closer to $N$; we shall see that the rank of $B$ is actually
crucial in low-rank approximation calculations later on. 

\section{Lyapunov Equation - Algebraic Reduction}
\label{sec:Lyapunov}

Recall that we have now considered a spatial discretization of the initial SPDE and we
have localized the problem near a locally deterministically stable steady state via 
linearization. We calculated upper bounds on the discretization error and on the 
linearization error for a finite time scale, including the noise truncation error. 
However, although we now can work with the linear SODE problem 
\be
\label{eq:brSODElinU}
\txtd U = \cA U~\txtd t+B~\txtd W^R,\qquad U(0)=U_0,
\ee
we are still surprisingly \emph{far from a practical computable problem for many 
applications!} It is well-known~\cite{BerglundGentz,KuehnSDEcont1} that 
$V(t):=\Cov(U(t))$ satisfies the matrix ordinary differential equation (ODE) 
given by
\be
\label{eq:Lyatemp}
\frac{\txtd }{\txtd t} V= \cA V+V\cA^\top + BB^\top=:L_{\cA}V + \cB.
\ee
The stationary problem is given by
\be
\label{eq:Lya}
0= \cA V+V\cA^\top + \cB.
\ee
It is well-known~\cite{Antoulas}, that under the assumption~\ref{a:spectrum}, 
there exists a unique stationary solution $V_*$ to~\eqref{eq:Lya}, which is 
stable for the time-dependent problem~\eqref{eq:Lyatemp}. However, we work on 
a finite-time interval $\cI$ so we need a convergence rate.

\begin{thm}
\label{thm:goal3}
Suppose~\ref{a:spectrum} holds so that $\cA$ also has a spectral gap then 
there exists a constant $C_\tau>0$ such that
\be
\label{eq:decayV}
\|V(t)-V_*\|_2\leq C_\tau (\|V(0)-V_*\|_2)
\txte^{-t\min_i(|\textnormal{Re}(\lambda_i)|)}
\ee
or, alternatively
\be
\label{eq:decayV2}
\|V(t)-V_*\|_2\leq C_\tau(H) ~(\|V(0)-V_*\|_2)\txte^{-2t/\|H\|_2},
\ee
where $H$ solves $\cA H+H\cA^\top+2\Id=0$. 
\end{thm}

\begin{proof}
For \eqref{eq:decayV}, one first uses that the eigenvalues of the linear 
operator $L_\cA$ are given by $\lambda_i+\lambda_j$ where 
$\{\lambda_k\}_{k=1}^N$ are the eigenvalues of $\cA$~\cite{Bellman}.
For~\eqref{eq:decayV2}, we use that $V_*$ is a steady state of~\eqref{eq:Lyatemp}
to obtain
\begin{align*}
 V(t)&=V_*+\mathrm{exp}(t \cA)(V(0)-V_*)\mathrm{exp}(t\cA^\top)\\
 \Rightarrow\quad \|V(t)-V_*\|_2&\leq \|\mathrm{exp}(\cA t)\|_2^2\|V(0)-V_*\|_2
\end{align*}
Then by \cite[Theorem 3.1]{HuLiu}
\begin{align*}
 \|\mathrm{exp}(t\cA)\|^2_2\leq C_\tau(H)~
\mathrm{exp}\left(\frac{-2t}{\|H\|_2}\right),
\end{align*}
which leads to \eqref{eq:decayV2}.
\end{proof}

Basically, Theorem~\ref{thm:goal3} gives an estimate, which allows us to 
reduce the computation to the Lyapunov equation~\eqref{eq:Lya} to the stationary
case as long we do not start far away from the locally linearized approximate
solution. Unfortunately, direct solution methods, such as the Bartels-Stewart 
algorithm~\cite{BartelsStewart} are unlikely to work as the space dimension 
$N$ grows drastically in practice as we decrease $h$. Furthermore, $N$ increases 
due to the curse of dimensionality as $d$ increases. Direct solution methods 
come with a complexity $\cO(N^3)$ and with storage requirements of $\cO(N^2)$, 
which limits their applicability to problems of moderate sizes. Therefore, 
we have to use special methods for large-scale Lyapunov equation 
that work with complexities and storage requirements linear in $N$.  

\section{Lyapunov Equation - Low-Rank and Computation}
\label{sec:iteration}

In this section, we want to consider and numerically compute a low-rank 
approximation of $V_*$. This involves two steps, which can be be accomplished 
by carefully tracing the literature: (1) understanding error estimates for the 
low-rank approximation and (2) finding an error estimate for the computation
in a low-rank iterative algorithm for solving Lyapunov equations.

\subsection{Low-rank Approximations and Singular Value Decay}
\label{ssec:lowrank}

Consider the singular value decomposition of $V_*$:
\begin{align*}
 V_*&=Y\Sigma X^\top, \quad Y^\top Y=\Id=X^\top X,\\
 \Sigma&=\mathrm{diag}(\sigma_1(V_*),\ldots,\sigma_N(V_*)),\\
 \sigma_1(V_*)&\geq\ldots\geq\sigma_{\ell}(V_*)>\sigma_{\ell+1}(V_*)=\sigma_n(V_*)=0,
\end{align*}
where $\ell=\mathrm{rank}(V_*)$. The best approximation of $V_*$ of rank $\fr$ is,
by the Eckhart-Young theorem \cite[Theorem 2.4.8.]{GolubVanLoan}, obtained by
\be
\label{Eq:lowrankSVD}
V_*\approx V^{\textnormal{lr},\fr}_*
:=\sum\limits_{i=1}^{\mathfrak{r}} \sigma_iu_ix_i^\top=Y_\fr\Sigma_\fr X_\fr,
\ee
where $Y_\fr,~X_\fr$ contain the first $\fr$ columns (left and right singular 
vectors) of $Y,X$, and $\Sigma_\fr$ the $\fr$ largest singular
values. Note that since $V_*=V_*^\top $, $X_\fr=Y_\fr$ can be chosen.
The approximation error is given by
\be
\|V_*-V^{\textnormal{lr},\fr}_*\|_2\leq \sigma_{\fr+1}.
\ee
If the singular values of $V_*$ decay rapidly towards zero, a small error 
can be achieved with small values of $\fr$. In fact, solutions of large-scale 
Lyapunov equations with low-rank inhomogeneities often show a fast singular 
value decay. This phenomenon has been theoretically investigated, e.g., 
in~\cite{Penzl,Grasedyck,AntoulasSorensenZhou,Sabino,Beckermann2016}. By 
assumption~\ref{a:operator} we restrict to the symmetric case $\cA=\cA^\top$. 
The following basic estimate on the singular value decay
can be found in~\cite{Penzl}
\begin{align}
\label{Eq:sigmabound_Penzl}
 \sigma_{R\fr+1}(V_*)\leq \sigma_{1}(V_*)
\left(\prod\limits_{i=0}^{{\mathfrak{r}}-1}
\frac{\kappa(\cA)^{(2i+1)/(2{\mathfrak{r}})}-1}
{\kappa(\cA)^{(2i+1)/(2{\mathfrak{r}})}+1}\right)^2,
\quad 1\leq R\fr\leq N,
\end{align}
where $\kappa(\cA)=\|\cA\|_2\|\cA^{-1}\|_2$. A more precise bound is 
developed in~\cite{Sabino} using $\textnormal{spec}(\cA)\subset[a,b]$ 
and quantities related to elliptic functions and integrals, which we 
briefly recall in the following definition.

\begin{defn}
\label{Def:ellipint}
The \textit{elliptic integral of the first kind} defined on $[0,1]$ 
with respect to the \textit{modulus} $0<k<1$ is
\begin{align*}
 s_k(x):=\int_0^x\frac{1}{\sqrt{(1-t^2)(1-k^2t^2)}}~\txtd t.
\end{align*}
Define also the elliptic functions $\mathrm{sn}, \mathrm{dn}$ by
\begin{align*}
 \mathrm{sn}(s_k)=x(s_k),\quad \mathrm{dn}(s_k)=\sqrt{1-k^2\mathrm{sn}(s_k)}.
\end{align*}
The associated \textit{complete elliptic integral (w.r.t. the modulus 
$k$)} is the value $K:=s_k(1)$. The \textit{complementary modulus} is 
$k'=\sqrt{1-k^2}$ and the associated \textit{complementary complete elliptic
integral} by $K'=s_{k'}(1)$.
The \textit{nome} $q$ is defined by $q:=\exp{(-\pi K'/K)}$. 
\end{defn}

\begin{lem}(\cite{Lawden})
The nome $q$ and complementary modulus $k'$ also satisfy the 
identity~
\begin{align}
\label{Eq:modulusnome}
 \sqrt{k'}=\frac{1-2q+2q^4-2q^9+\ldots}{1+2q+2q^4+2q^9+\ldots}.
\end{align}
\end{lem}

Using these quantities, the singular values of $V_*$ can be bounded 
by the following result:

\begin{thm}(\cite[Theorem 2.1.1.]{Sabino},\cite{EllnerWachspress})
\label{thm:svddecay}
Let $\cA=\cA^\top$ and $\textnormal{spec}(\cA)\in[a,b]$ with $a:=\min\lambda_i<b
:=\max\lambda_i<0$. Set the complementary modulus to $k'=b/a=1/\kappa(\cA)$ 
and set $k$, the complete elliptic integrals $K'$ and $K$, as well as the 
nome $q$ via the relations in Definition~\ref{Def:ellipint}. Then it holds 
for the singular values of the solution of~\eqref{eq:Lya}
\begin{align}
\label{Eq:sigmabound_Sabino}
 \sigma_{R\fr+1}(V_*)\leq \sigma_{1}(V_*)\left(
\frac{1-\sqrt{k'_\fr}}{1+\sqrt{k'_\fr}}\right)^2,\quad 1\leq R\mathfrak{r}\leq N,
\end{align}
where $k_\fr'$ relates to $q^\fr$ in the same way $k'$ is build from 
$q$ via~\eqref{Eq:modulusnome}:
\begin{align*}
\sqrt{k_\fr'}=\frac{1-2q^\fr+2q^{4\fr}-2q^{9\fr}+\ldots}
{1+2q^\fr+2q^{4\fr}+2q^{9\fr}+\ldots}.
\end{align*}
\end{thm}

We stress that, while better than~\eqref{Eq:sigmabound_Penzl}, the 
bound~\eqref{Eq:sigmabound_Sabino} might not be very sharp in practice,
where one often observes an even faster singular value decay. One reason is 
that Theorem~\ref{thm:svddecay} only uses $\kappa(\cA)$, i.e., the extremal 
eigenvalue $a,b$ of $\cA$. More realistic, but also more difficult to compute, 
bounds can be obtained by using more than these two eigenvalue of 
$\cA$~\cite{Sabino}.\medskip 

{\small \textbf{Remark:} \textit{In case of non-symmetric $\cA$, the above bounds 
are not applicable. In this case the singular value decay of $V_*$ is more 
complicated and we refer to e.g., the discussions in~\cite{BakerEmbreeSabino,
Beckermann2016,Sabino}.}}\medskip

Judging by~\eqref{Eq:sigmabound_Sabino}, the decay of the singular values depends 
mainly on $\kappa(\cA)$ and the value $R=\mathrm{coldim}(B)$. For instance for 
fixed $a$, the closer $b$ to the origin, i.e.~the smaller the spectral gap, the 
larger $\kappa(\cA)$ and, consequently, the closer $k$, $k'$ will be to one 
and, respectively, zero. Hence, $K'$ and $K$ will tend towards $\frac{\pi}{2}$ 
and $\infty$, respectively, leading to $q$ close to one. In this extreme 
situation, $q^\fr$ will also be close to one, leading, in the end, to 
\benn
\frac{1-\sqrt{k'_\fr}}{1+\sqrt{k'_\fr}}\approx 1 
\eenn
such that no singular value decay might be observed. In particular, the main 
insight is that the spectral gap also plays a \emph{numerical analysis} role
at the error in step (S5). Furthermore, the column dimension $R$ of $B$ plays 
a significant role in the bound~\eqref{Eq:sigmabound_Sabino}. Recall the $B$ 
is the result from evaluating $g$ at $u^*$ and approximating the stochastic 
process $W(t)$. Recall that in (S3) we truncated the $Q$-Wiener process~\cite{DaPratoZabczyk} 
to obtain a numerical approximation~\cite{LordPowellShardlow}
\begin{align*}
W(t)\approx\sum\limits_{i=1}^R\sqrt{\lambda_{Q,i}}~\beta_i(t)~e_i
\end{align*}
for a basis $\{e_i\}_{i=1}^\I$ of $\cH$, eigenvalues $\lambda_{Q,i}$ 
of $Q$, and independent Brownian motions $\beta_i(t)$. 
The scalars $\lambda_{Q,i}$ form a non-increasing sequence. Hence, the 
slower the $\lambda_{Q,i}$ decrease, the higher the value of $R$ should be
chosen and, consequently, the slower the decay of the singular values of $V_*$.
This implies the key practical conjecture that space-time white noise would be 
more difficult to treat numerically than a $Q$-trace-class Wiener process
in our setting.
  
\subsection{Error Bounds for Numerically Computed Low-Rank Solutions}

Using the low-rank approximation~\eqref{Eq:lowrankSVD} is not practical as it 
requires to first obtain $V_*$ and then compute its singular valued decomposition. 
Numerical methods for large-scale Lyapunov equations~\cite{Simoncini,BennerSaak} 
typically directly compute low-rank factors $\cZ\in\R^{N\times \fr}$, $\fr\ll N$ that 
form a low-rank approximate solution $V^{\textnormal{lr},\fr}_*=\cZ\cZ^\top$ which 
will, however, not be optimal in the sense of the SVD based approximation~\eqref{Eq:lowrankSVD} 
but close if the method is properly executed. The advantage of these methods is 
that they are able to provide accurate low-rank solutions in a very efficient manner 
by utilizing tools from large-scale numerical linear algebra. By exploiting, e.g., 
the sparsity of $\cA$ and the low-rank $R$ of the inhomogeneity, state-of-the-art 
methods~\cite{Simoncini,BennerSaak} are able to compute low-rank solution factors 
at complexities and memory requirements of $\cO(N)$.\medskip

One iterative method for solving \eqref{eq:Lya} is based on the fact that for 
any $\alpha\in\C_-$,
\eqref{eq:Lya} is equivalent to the matrix equation
\begin{align*}
X&=\cC(\alpha)X\cC(\alpha)+\cB(\alpha)\cB(\alpha)^H\\
\text{with}\quad\cC(\alpha)&:=(\cA+\alpha \Id)^{-1} 
(\cA-\overline{\alpha}\Id),\quad \cB(\alpha):=
\sqrt{-2\textnormal{Re}(\alpha)}(\cA+\alpha \Id)^{-1}B,
\end{align*}
where $(\cdot)^H$ is the Hermitian conjugate.
This motivates the self-evident iteration scheme 
\begin{align*}
 X_j=\cC(\alpha_j)X_{j-1}\cC(\alpha_j)+\cB(\alpha_j)\cB(\alpha_j)^H
\end{align*}
for varying $\alpha_j\in\C_-$, which is the alternating directions implicit 
(ADI) iteration for Lyapunov equations~\cite{Wachspress1}. In order to be 
applicable to large-scale equations, one uses $X_0=0$ and, after a series of 
basic algebraic manipulation~\cite{LiWhite}, one arrives at the 
low-rank ADI (LR-ADI) iteration~\cite{LiWhite,BennerKuerschnerSaak,Kuerschner,Kuerschner2}
\begin{align*}
 H_j&=(\cA+\alpha_j\Id)^{-1}\cW_{j-1},\quad \cW_j=\cW_{j-1}-2\textnormal{Re}(\alpha_j)H_j,\\
 \cZ_j&=[\cZ_{j-1},~\sqrt{-2\textnormal{Re}(\alpha_j)}H_j]\in\C^{N\times Rj}
\end{align*}
for $\cW_0:=B$, $j\geq 1$. It produces low-rank approximations of the solution 
of~\eqref{eq:Lya} of the form $V_*\approx V_j:=\cZ_j\cZ_j^H$. 
The numbers
$\alpha_j\in\C_-$ are referred to as shift parameters and are crucial for a fast 
convergence  of the LR-ADI iteration. Obviously, the main numerical effort of the 
LR-ADI iteration comes from the solution of the linear systems of equations
$(\cA+\alpha_j\Id)H_j=W_{j-1}$ for $H_j$ which can for sparse $\cA$ be done 
efficiently by sparse-direct~\cite{DuffErismanReid} or iterative solvers~\cite{Saad}.
However, the number of right hand sides in each linear systems is given by the key quantity $R$. We see here that the higher $R$, the larger the numerical effort of the LR-ADI iteration becomes. 
The error of the constructed low-rank approximation $V_j$ is given by 
\begin{align}
\label{adierror_lyap}
 V_j-V_*=\cJ_jV_*\cJ_j^H,\quad
\cJ_j:=\prod\limits_{i=1}^{j}\cC_i,\quad \cC_i:=\cC(\alpha_i)
\end{align}
such that
\begin{align*}
 \|V_j-V_*\|_2&\leq \kappa(S)^2\|V_*\|_2R^2_j,\quad 
\Theta_j:=\prod\limits_{i=1}^{j}\rho_i,\\
 \rho_i&=\rho(\cC_i)=\max_{z\in\textnormal{spec}(\cA)}
\left|\frac{z-\overline{\alpha_i}}{z+\alpha_i}\right|,
\end{align*}
and $S$ is the matrix containing the eigenvectors of $\cA$. Since 
$\textnormal{spec}(\cA)\in\C_-$ and $\alpha_i\in\C_-$, it holds $\rho_i<1$, 
$\forall i\geq 1$ and, thus, $\Theta_j=\rho_j\Theta_{j-1}<\Theta_{j-1}$, 
indicating that the sequence of the spectral radii $\Theta_j$ is monotonically 
decreasing and, in the limit, will approach the value zero. The shifts $\alpha_i$ 
should therefore be chosen such that $\Theta_j$ is as small as possible leading 
to the ADI parameter problem
\begin{align}
\label{adiminmax}
 \lbrace \alpha^*_1,\ldots,\alpha^*_j\rbrace = \mathrm{arg}\min
\Theta_j=\mathrm{arg}\min\limits_{\alpha_i\in\C_-}
\max\limits_{z\in\textnormal{spec}(\cA)}\left|\prod\limits_{i=1}^{j}
\frac{z-\overline{\alpha_i}}{z+\alpha_i}\right|.
\end{align}
This problem is in general a formidable task which has been addressed in numerous works, 
e.g., in~\cite{EllnerWachspress,Wachspress1,Sabino,BennerKuerschnerSaak,Kuerschner}. 
Again, the situation simplifies for the important case $\cA=\cA^\top$, where real 
shifts are usually sufficient (and $\cZ_j$ will also be real). Note that in this 
case $\kappa(S)=1$. In summary, we have the following relevant result for our
purposes:

\begin{thm}(\cite{Wachspress2,Sabino,Wachspress1})
\label{thm:goal4}
With the same assumptions and settings for $k$, $k'$, $K$, $K'$, $q$ as 
in Theorem~\ref{thm:svddecay}, construct real shifts parameters
$\alpha_1,\ldots,\alpha_j\in\R_-$ by 
 \begin{align}\label{optishift}
 \alpha_i=a\mathrm{dn}((2i-1)K/2j),~1\leq i\leq j 
\end{align}
with the elliptic function $\mathrm{dn}$ from Definition~\ref{Def:ellipint}. 
Using these shifts, the smallest value of the spectral radius 
$\Theta_j=\Theta_j(\alpha_1,\ldots,\alpha_j)$ in~\eqref{adiminmax} is 
$\frac{1-\sqrt{k'_j}}{1+\sqrt{k'_j}}$, where the
modulus $k'_j$ is associated to the nome $q^j$ via~\eqref{Eq:modulusnome}. Hence,
carrying out $j$ steps of the LR-ADI iteration using \eqref{optishift} yields
\begin{align*}
 \|V_j-V_*\|_2&\leq\|V_*\|_2\left(\frac{1-\sqrt{k'_j}}{1+\sqrt{k'_j}}\right)^2.
\end{align*}
\end{thm}

We can, thus, expect to approximate the solution $V_*$ by the LR-ADI low-rank 
approximation $V_j$ at a speed similar to the predicted singular value decay by 
Theorem~\ref{thm:svddecay}. Other low-rank algorithms for solving \eqref{eq:Lya}, 
e.g., rational Krylov subspace methods allow similar error 
bounds~\cite{DruskinKnizhnermanSimoncini,Beckermann2012}.

\section{Summary and Main Result}
\label{sec:result}

Although we stated and proved all our main error estimates, it is helpful to summarize 
the results to provide a CERES. Recall from the introduction that our setup considered 
four steps. To combine the four steps, we have to link operators on $\dot{\cH}^1$ with 
the finite-dimensional approximation spaces $\cS_h$. If $L_h:\cS_h\ra \cS_h$ is a 
finite-dimensional linear operator, then we can always view it as an infinite-dimensional 
linear operator $L$ on $\dot{\cH}^1$ by declaring basis vectors
not in $P_h \dot{\cH}^1$ to lie in $\textnormal{nullspace}(L)$.

\begin{thm}
\label{thm:main}
Suppose \ref{a:operator}-\ref{a:smallnoise} hold. Let $\textnormal{Cov}(u)$ denote the 
covariance operator of the SPDE~\eqref{eq:SPDE}. Then a low-rank solution $V_j$ of the 
locally linearized discretized problem on $\cS_h$ near the steady state $u^*$ computed 
after $j$ ADI steps satisfies the CERES
\be
\label{eq:goalfinal}
\sup_{t\in[0,T]}\|\textnormal{Cov}(u)-V_j\|_{\cL_2}\lesssim\sup_{t\in\cI}\left[ 
\textnormal{err}_{\textnormal{(S1)}}
+\textnormal{err}_{\textnormal{(S2)}}+\textnormal{err}_{\textnormal{(S3)}}
+\textnormal{err}_{\textnormal{(S4)}}+\textnormal{err}_{\textnormal{(S5)}}\right]
\ee 
and the individual error terms are given by
\bea
\textnormal{err}_{\textnormal{(S1)}}&=& C_\txtd h^{1+r},\\ 
\textnormal{err}_{\textnormal{(S2)}}&=&  
\eta^*_\cA+C_\txtl [\|\Cov(u_h(0))-\Cov(\tilde{U}_h(0))\|_2] \txte^{-t\min_i |\cA_{ii}|},\\ 
\textnormal{err}_{\textnormal{(S3)}}&=& C_{tr} \sum_{i=R+1}^N \lambda_{Q,i}\\ 
\textnormal{err}_{\textnormal{(S4)}}&=& 
C_\tau (\|\Cov(U_h(0))-V_*\|_2)\txte^{t\max(\textnormal{spec}(\cA))}\\
\textnormal{err}_{\textnormal{(S5)}}&=& 
\|V_*\|_2 \left(1-\sqrt{k'_j}\right)^2/\left(1+\sqrt{k'_j}\right)^2,
\eea
where $C_\txtd,C_\txtl,C_\tau,k_j',C_{tr},\eta^*_\cA>0$ are constants depending upon the data 
(i.e.~on~$A,f,g,Q$), the terms $u_h(0),\tilde{U}_h(0)$ are initial conditions for the discretized
full and linearized problems, $\lambda_{Q,i}$ are eigenvalues of $Q$, $R\in\N$ is the noise 
truncation level, $V_*=\lim_{t\ra +\I}U_h(t)$ is the finite asymptotic 
limit for the stationary problem, and $\cA$ is the leading-order discretized linear 
operator part near $u^*$.
\end{thm}
 
\begin{proof}
Just applying a triangle inequality to the left-hand side of~\eqref{eq:goalfinal}, the
proof follows from a direct application of 
Theorems~\ref{thm:goal1},~\ref{thm:goal2},~\ref{thm:truncate}, ~\ref{thm:goal3}, 
and~\ref{thm:goal4}.
\end{proof}

Theorem~\ref{thm:main} illustrates very well that it would be \emph{short-sighted} to just
look at one source of error. For example, even if the spatial discretization $h$ is
extremely small, the actual error could be very large if the spectral gap is small, i.e.,
the deterministic steady state is only weakly locally stable. We may also summarize the
behaviour of the different constants and terms in the CERES into more practical 
observations, which effects lead to \emph{smaller error}:

\begin{itemize}
 \item A large gap in the spectrum of the local linearization of the deterministic
part exists.
 \item A small spatial discretization level $h$ is chosen.
 \item One starts with initial data close to the local approximating OU stationary 
distribution.
 \item The nonlinear part of $f$ does not have a strong effect on sub-exponential time scales.
 \item The noise is small enough to stay for a long time in the sub-exponential regime.
 \item The $Q$-trace-class operator has fast decaying eigenvalues.
 \item The iteration number $j$ in the low-rank Lyapunov solver is chosen large.  
\end{itemize} 

For a detailed discussion of these effects, we refer to the individual proofs of the 
different parts of the CERES in previous sections. However, it is very interesting
to note that certain effects, which decrease the error occur in \emph{multiple steps}.
Obviously this is true for the spectral gap of $\cA=A_h+P_h\txtD_uf(u^*)$ in (S2) and 
(S4) but also for the type of noise, which crucially influences the, usually growing, function 
$\eta^*(t)$ in (S2) as well as the convergence rate of a low-rank approximation in (S5).
In addition, note that we have observed that several errors are \emph{linked} and cannot
be treated independently! For example, a small column dimension $R$ guarantees a good 
low rank approximation in (S5) but selecting $R$ small gives a larger error for the noise 
truncation in (S3).
 
\section{Outlook}
\label{sec:outlook}

We stress again that our results presented here should be viewed as a first key step to
introduce a general framework of CERES for high-dimensional stochastic problems, where
many different sources of error naturally occur. In this regard, a multitude of 
problems can, and should, now be tackled from a similar viewpoint. For example, 
uncertainty quantification of random partial differential 
equations (RPDE)~\cite{GhanemSpanos,XiuHesthaven,LemaitreKnio,NobileTemponeWebster} 
contains an entire chain of error sources such as truncation error for polynomial 
chaos, dynamical error if the RPDE is just an approximation, error from the large-scale
numerical linear algebra, and/or error due to reduced bases, just to name a few. Hence,
to compute a CERES in a single norm, such as the spectral norm we used here, for all
the steps would be very worthwhile. Similar issues also appear for problems involving
large deviations and transition paths in high-dimensional energy 
landscapes~\cite{BolhuisChandlerDellagoGeissler,ERenVanden-Eijnden,
LelievreStoltzRousset,HenkelmanUberuagaJonsson}, where developing a CERES would
definitely be very helpful. 

On the very concrete level of the SPDE~\eqref{eq:SPDE} studied here, several interesting 
directions could be pursued. Firstly, we
do not claim that all our estimates are sharp and/or the assumptions are the 
theoretically weakest possible. Already the CERES presented in this work is interesting 
and difficult to chain together properly from its different components. Nevertheless,
improvements might be possible, e.g., if $f$ is a strongly dissipative non-Lipschitz
nonlinearity, we expect that the results still hold from a dynamical perspective but
essentially the steps (S1)-(S2) would then require a major extension or even a 
completely new approach. 

Furthermore, it could be desirable to specify the constant
$C_\txtd$ precisely as the techniques in~\cite{Kruse,KruseLarsson} are more
explicit than the final statements on optimal regularity.
Unfortunately this would entail re-writing the entire optimal regularity proof
for the spatial discretization so we refrain from attempting to carry this program in
this work. Similar remarks also apply to other constants, which are expected to be
of moderate/non-asymptotic relevance only in many practical applications anyhow. 

\medskip
\textbf{Acknowledgments:} CK would like to thank the VolkswagenStiftung for
support via a Lichtenberg Professorship. Furthermore, CK would like to thank 
Daniele Castellano for interesting discussions regarding moment equations.
The predominant part of work on this article was done while PK was affiliated 
with the Max Planck Institute for Dynamics of Complex Technical Systems in Magdeburg, 
Germany.


\end{document}